\documentclass[12pt]{amsart}
\usepackage{amsmath}
\usepackage{amssymb}
\usepackage{amsfonts}
\usepackage{amsthm}
\usepackage{verbatim}
\usepackage{amscd}
\usepackage{cite}
\usepackage{leftidx}
\usepackage{enumerate}
\usepackage{txfonts}
\usepackage{manfnt}
\usepackage{amscd}
\usepackage[mathscr]{eucal}
\usepackage{hyperref}
\usepackage{graphicx}
\usepackage{palatino}
\usepackage[utf8]{inputenc}
\def\gG{\mathbb G}
\def\fF{\mathbb F}
\def\sm{\mathrm{sm}}

\newtheorem{thm}{Theorem} 

\newtheorem*{thm*}{Theorem}
\newtheorem*{prop*}{Proposition}
\newtheorem{cor}[thm]{Corollary}
\newtheorem*{cor*}{Corollary}

\newtheorem{lem}[thm]{Lemma}
\newtheorem*{lem*}{Lemma}

\newtheorem*{claim*}{Claim}
\newtheorem{prop}[thm]{Proposition}

\newtheorem{defn}[thm]{Definition}
\theoremstyle{remark}
\newtheorem*{n&c}{Notations and conventions}

\newtheorem{rem}[thm]{Remark}
\newtheorem*{rem*}{Remark}
\newtheorem{crit-rem}[thm]{Critical remark}

\newtheorem{example}[thm]{Example}
\newtheorem*{example*}{Example}

\newtheorem*{defn*}{Definition}

\def\inv{^{-1}}

\DeclareMathOperator{\del}{\partial}
 \DeclareMathOperator{\Hom}{Hom}

\def\gG{\mathbb G}

\def\refp #1.{(\ref{#1})}
\newcommand\carets [1]{\langle #1 \rangle}

\newcommand{\A}{\mathcal{A}}

\newcommand{\X}{\mathcal{X}}

\newcommand{\Cal}[1]{\mathcal #1}

\newcommand{\ul}[1]{\underline {#1}}

\def\sbr #1.{^{[#1]}}
\def\sfl #1.{^{\lfloor #1\rfloor}}

\def\what{\widehat}
\def\inv{^{-1}}
\def\?{{\bf{??}}}

\def\A{\Bbb A}
\def\X{\Cal X}

\def\C{\mathbb C}
\def\P{\mathbb P}

\def\ord{\text{\rm ord}}

\def\O{\mathcal O}

\def\Sym{\textrm{Sym}}

\def\g{\mathfrak g}

\def\m{\mathfrak m}

\def\1/2{\frac{1}{2}}

\def\d{\mathrm{d}}

\def\im{\text{im}}

\def\2{{[2]}}
\def\l{\ell}
\def\nl{\newline}

\def\<{\langle}
\def\>{\rangle}

\def\im{\text{im}}
\def\2{{[2]}}
\def\l{\ell}

\def\scl #1.{^{\lceil#1\rceil}}
\def\spr #1.{^{(#1)}}
\def\sbc #1.{^{\{#1\}}}

\def\subpr#1.{_{(#1)}}

\def\beq{\begin{equation*}}
\def\eeq{\end{equation*}}

\def\g3{{\Gamma\spr 3.}}

\newcommand{\eqspl}[2]{
\begin{equation}\label{#1}
\begin{split}
#2\end{split}\end{equation}}

\newcommand{\beginalphaenum}{
\begin{enumerate}\renewcommand{\labelenumi}{ }
\item \begin{enumerate}
}

\def\eex{\end{rm}\end{example}}

\begin{document} 
\title{Gaussian scrolls, Gaussian flags and duality}
\author 
{Ziv Ran}


\thanks{arxiv.org 1909.03307; Rendiconti Circolo Mat. Palermo}
\date {\today}


\address {\nl UCR Math Dept. \nl
Skye Surge Facility, Aberdeen-Inverness Road
\nl
Riverside CA 92521 US\nl 
ziv.ran @  ucr.edu\nl
}

 \subjclass[2010]{14n05}

\begin{abstract}
A projective variety whose Gauss map has positive dimensional fibres corresponds to
a special kind of scroll called \emph{Gaussian}. A Gaussian scroll is a member of a canonical
derived \emph{ Gaussian flag}. 
	We introduce a duality in the class of Gaussian  scrolls and flags
	and study its consequences. 
	In particular, a Gaussian scroll is dual to the derived or tangent developable
	scroll of
	a Gaussian scroll  in the dual projective space, and is the 'leading edge' or antiderived
	scroll of its derived stationary scroll.\par
	\end{abstract}

\maketitle
\section{Introduction: Gauss-deficient varieties}
\subsection{Setup: degenerate Gauss maps as stationary scrolls}
Let $X$ be an irreducible closed
$n$-dimensional  subvariety of $\P^N$ 
(always over $\C$). In what follows we shall assume-
without loss of interesting generality- that $X$
is nondegenerate and not a cone
The Gauss map \[\gamma_X:X_\sm\to\gG=\gG(n,N)\]  
is the morphism (rational map on $X$)  that  maps
a smooth point $x\in X$ to the embedded tangent space 
$\tilde T_xX$, which is an $n$-plane in $\P^N$. 
By a well-known theorem of Zak \cite{zak},
this map is finite whenever $X$ is smooth. However there exist many
examples of singular varieties $X$ for which $\gamma_X$ is  not
generically finite to its image. These include cones, joins,  tangent developables and more, 
see \cite{griffiths-harris-ldg}, \cite{akivis-gauss}, 
\cite{akivis-goldberg}, \cite{ran-gauss}
and below. In this case  $X$ is said
to have a \emph{degenerate} Gauss
mapping and $g=\dim(\gamma_X(X))$ is called the \emph{Gauss dimension} of
$X$ while the difference $\dim(X)-g$ is
called the \emph{Gauss deficiency} of $X$ and $X$ is said to be \emph{Gauss-deficient}
if its Gauss deficiency is $>0$. In fact, the fibres of $\gamma_X$
are known (see op. cit. or below) to be 
(open subsets of) linear spaces.
Accordingly, a Gauss-deficient variety
determines a particular kind
of (singular) scroll which we will call \emph{stationary scroll}, where
stationary means that the Gauss map is constant on rulings.
The stationary scrolls arising from Gauss maps are called \emph{Gaussian}.\par
The main purpose of this paper is to introduce a scroll-centric viewpoint on Gauss-deficient varieties.
Namely, a Gauss-deficient variety determines, besides its Gaussian scroll,
a canonical  maximal \emph{Gaussian flag}
of Gaussian scrolls, in which every member except the first is the tangent developable of the preceding,
and every member except the last is the 'leading edge' of the succeeding; moreover, there is a duality
in the class of Gaussian flags, with 'developable' dual to 'leading edge' (see Theorem \ref{duality-thm}).
We elaborate below.
 \par
 \subsection{Known results}
In \cite{griffiths-harris-ldg}, 
Griffiths and Harris state a structure theorem for Gauss-deficient varieties. 
They assert somewhat vaguely that such varieties
are 'built up from' (actually foliated by) cones and developable varieties. 
(\cite{griffiths-harris-ldg}, p. 392). 
 In the case $n = 2$ (more generally, $g=1$) they show
 that any stationary scroll is in fact a cone or a developable.
 Varieties with Gauss dimension $g=2$ were classified by Piontkowski
 \cite{piontkowski}.
In \cite{akivis-gauss}
\footnote{We thank Dr. L. Song for bringing the paper \cite{akivis-gauss}
	to our attention.}, Akivis, Goldberg and Landsberg
 present  some examples, especially "unions of conjugate spaces",
  which are neither cones nor developables. 
Subsequently Akivis and Goldberg \cite{akivis-goldberg-2000}
refined the Griffiths-Harris analysis to show that stationary scrolls
  are actually built up from, i.e. foliated by, certain 3 basic types.
  Further results are due to E. Mezzetti and O. Tommasi
  \cite{mezzetti2002} and \cite{mezzetti2004}.
  See also \cite{zak}, \cite{landsberg-lectures} and \cite{akivis-goldberg} for
  comprehensive accounts of results on Gauss maps and various types
  of degeneracy. 
  Still, the question of construction and classification
of all stationary scrolls appears to be an open problem.
\subsection{Contents of the paper}
In \S\ref{stationary-sec} we define some notions related to the sort of scrolls
that appear in the structural study of Gauss maps.  This allows us to give an informal 
statement of our main duality result (see \S \ref{duality-preview-sec}). 
In \S \ref{construction} and \S \ref{properties-sec} 
we review and slightly extend
some known constructions and properties for Gauss-deficient 
varieties and some related notions such as second fundamental form
(derivative of the Gauss map). In \S \ref{duality} we state and prove 7our main result,
establishing a duality for Gaussian flags. Finally in \S \ref{small} we give some results
on varieties with small Gauss dimension.
\subsubsection*{Acknowledgment}
We thank Professor E. Mezzetti and Professor O. Tommasi for helpful comments
and references.
\section{Gaussian scrolls and flags}\label{stationary-sec}
\subsection{Stationary and Gaussian scrolls: definitions}\label{stationary-gaussian} 
Considering that stationary scrolls are necessarily singular,
it will be convenient to adopt a slightly more general viewpoint
and talk about \emph{parametric scrolls}: 
\begin{defn}
A parametric
scroll consists of a \par
(a) $\P^k$-bundle $X/B$; \par
(b)  a generically finite
morphism $f:X\to\P^N$, called the spreading map, that restricts to a linear isomorphism
of a general fibre of $X/B$ with a linear $\P^k$ in $\P^N$.
\end{defn} 
We allow the trivial case $k=0$.\par
For $x\in X$ we denote by $\tilde T_xX$ the unique linear subspace of $\P^N$
containing $f(x)$ which has tangent space $df_x(T_xX)$.
This is called the projective tangent space.\par
To a parametric scroll
is associated its  \emph{classifying morphism}
\[\phi: B\to\gG(k,N)\]
 (a priori the classifying map is just
a rational map but as our viewpoint is essentially local over $B$ 
no information is lost by base-changing it to a morphism).\par 
Note that $\phi$ is generically finite to its image
hence, by char. 0, generically unramified.
Next we define stationary scrolls:
\begin{defn}\label{stationary-defn}
A parametric scroll  $(X/B, f)$ 
is \emph{stationary} if either one of the following two equivalent conditions hold:\par
(a) 
for general $b\in B$ with fibre $X_b$ and general $x\in X_b$,  the projective tangent space $\tilde T_xX$, 
depends only on $b$;\par
(b) the Gauss (rational) map
\[\gamma_X:X\cdots\to\gG(n, N),\]
\[x\mapsto \tilde T_xX\] factors through a map
\eqspl{gammabar}{\bar\gamma_X:B\to\gG(n,N).}
.\end{defn} 
\begin{rem}(i)
All parametric scrolls with $k=0$ are stationary;
the stationary scrolls with $k>0$ yield
varieties with degenerate Gauss map.\par
(ii) the map $\bar\gamma_X$ need not be generically finite to its image even if the classifying map
is (Remark \ref{nonstrict}).
\end{rem}
For a stationary scroll $(X_B, f)$ as above,  
the image $\bar X=f(X)\subset\P^N$ of the spreading map is a variety with the
property that the fibre through a general point of its Gauss map $\gamma_{\bar X}$, 
which itself a $\P^m$, 
contains the ruling $f(\P^k_b)$, for all $b$ in the corresponding fibre of $\bar\gamma_X$.
\begin{defn}
The stationary scroll $(X_B, f)$ is said to be \emph{Gaussian} if $k=m$, i.e. if $X\to B$ is,
 up to generically finite base-change, the Gauss map of a variety $\bar X\subset\P^N$.
 \end{defn}
 Note that given a stationary scroll $(X_B, f)$, there is an associated Gaussian scroll
 $(X^+_{B^+,} f^+)$, namely a suitable resolution of the Gauss map of $f(X)$.
 This is called the \emph{saturation} of $(X_B, f)$ and $(X_B, f)$ is said to be \emph{saturated}
 if $X^+_{B^+}=X_B$ (in particular, $B^+=B$).  Note the diagram
 \eqspl{saturation}{
 \begin{matrix}
 X_B&\to& X^+_{B^+}\\
 \downarrow&&\downarrow\\
 B&\to&B^+
 \end{matrix}
 }
 with $X_B\to X^+_{B^+}$ injective and $B\to B^+$ surjective and $f(X)=f^+(X^+)$.\par
 For example if $X_B$ is the trivial scroll (fibre dimension 0, $B=X$), its saturation is
 the Gauss map $\gamma_X:X\to\gamma_X(X)$, viewed as a scroll.

\subsection{Derived scroll; Gaussian flags}\label{gaussian-flags-sec}
\begin{defn}
Given a stationary scroll $X_B$ as above, its  tangential or derived scroll
$X'_B$ or $X_B^{(1)}$  is the scroll corresponding to the map $\bar\gamma_X$ above \eqref{gammabar}.
\end{defn}
Now given a stationary scroll $X_B$, a general point of $X'_B$ has the form  $q=p+v, v\in T_pX$.
Then $T_qX'_B$ is generated mod $T_pX$ by elements of the form $d^2/dt^2|_{t=0}(\alpha(t))$
where $\alpha(t)$ is an arc in $X$ with $\alpha(0)=p, d/dt|_{t=0}(\alpha(t)=v$. Modulo $T_pX$, this is
independent of $q\in T_pX$. Thus, the derived scroll
 $X^{(1)}_{B^{(1)}}$ is itself a stationary scroll.
In fact, if if $X_B$ is a Gaussian scroll then so is $X'_B$: 
if we denote the fibres of $X_B$ by $\P^k_b$ and tangent spaces by $\P^n_b$, and pick 
$\alpha(t)$ is a parametric arc through a general point $b$ in $B$,
 and we set $\P^k(t)=\P^k_{\alpha(t)}$ then we have
\[(d/dt)|_{t=0}(\P^k(t))\neq 0\mod\P^k(0)\]
and similarly, unless $(d/td)^i|_{t=0}(\P^k(t))=\P^N$, we have
\[(d/dt)^{i+1}_{t=0}(\P^k(t))\neq 0\mod (d/td)^i|_{t=0}(\P^k(t)).\]
Indeed if the $i$th and $(i+1)$st derivatives agree then inductively all the higher
derivatives agree as well- but these ultimately must fill up $\P^N$. This implies that is $X_B$ is Gaussian,
then so is $X'_B$. \par
This allows us to define higher derived scrolls and Gaussian flags:



\begin{defn} \label{gaussian-def}(i)
Let $X_B$ be a stationary scroll and let  $i>0$ be such that  $X^{(i)}_B$ is not filling, i.e. does not map onto $\P^N$. Then the $(i+1)$-st
derived scroll is defined as
the derived scroll of $X^{(i)}_B$:
\eqspl{derived}{
X^{(i+1)}_B=(X^{(i)}_B)^{(1)}.
} 
The flag $X_B^{(\bullet)}$ is called the  \emph{osculating flag} of scrolls associated to $X_B$.\par
(ii) A flag of scrolls $X_B^\bullet$ is said to be Gaussian if each $X^i_B$ is a  Gaussian scroll and
$(X_B^i)^{(1)}=X_B^{(i+1)}$
\end{defn}
Thus a nondegenerate Gaussian scroll can be extended upwards to a Gaussian flag terminating
in a filling scroll (whose spreading map is surjective).
Note that the osculating flags (of scrolls) defined above include the (plain) \emph{osculating flags}
where $X/B$ is a trivial scroll (fibre dimension 0), so $X^{(i)}_B$ is just
its ordinary $i$th osculating  scroll of $X$.\par

\begin{rem}\label{nonstrict}
Even if $X_B$ is effective (but not Gaussian), its derived scroll $X'_B$ 
may not be effective. 
Indeed start with any stationary $\P^k$-scroll $Y^1_{D_1}$ 
and in a general $\P^k_{d_1}$ fibre choose a 1-parameter family $\P^{k-1}_{d, d_1}$ so that together
these constitute a $\P^{k-1}$-scroll with 2-dimensional base, say $Y_D$  so that $D$ fibres over $D_1$.
Then we get the base-changed scroll $Y^1_D$
and clearly $T_yY_D=T_yY^1_D$ at a general point so  $Y_D'=Y_D^1$ and $Y^1_D$ 
is not effective even though $Y_D$ is.
\end{rem}
\subsection{Duality: a preview}\label{duality-preview-sec}
For a subspace $A\subset\P^N$, we denote by $A^\perp$ the dual subspace of 
the dual projective space $\check\P^N$ (sometimes called 'linear dual' do avoid confusion with dual
variety).
Then our main result here can be stated informally as follows (see Theorem \ref{duality-thm}
for a precise statement).
\begin{thm*}
The operation $A\mapsto A^\perp$ is a duality between  the collections of Gaussian flags
in $\P^N$ and in $\check \P^N$.
\end{thm*}
Of course
the dual of $(X^{(\bullet)}_{B})$ is $(X^{(\bullet)\perp}_B)$ (in reverse order), 
mapping to the dual projective space $\check\P^N$. Here $X^\perp_B$
is the scroll over $B$ with fibres $X_b^\perp$ and likewise for
$X^{(i)\perp}_B$. As the terminology implies, 
the linear duality 
relationship is symmetric. \par
Now given a Gaussian pair
\[X_B\subset X_B^{(1)},\]
its dual
\[(X_B^{(1)})^\perp\subset X_B^\perp,\]
isn't obviously Gaussian but rather 'co-Gaussian', in the sense that
\[(X_B^{(1)})^\perp=(X_B^\perp)^{\perp(1)\perp}.\]
For any stationary scroll $Y_B$, we call 
the subscroll $Y_B^{\perp(1)\perp}$ its \emph{antiderived} or leading edge scroll
and denote it by $Y_B^{(-1)}$ (see \S \ref{antiderived-sec}). Thus, a big part
 of the content of the duality theorem is that a Gaussian scroll is the derived of its antiderived (when
 the latter is nonempty: \[Y_B=(Y_B^{(-1)})^{(1)}.\]
If $Y_B$ is not a cone, which we generally assume, 
formation of antiderived scroll can be iterated until the empty scroll is reached.
A Gaussian scroll can thus be extended both up and down to a
uniquely determined maximal (in both directions) Gaussian flag, whose top member is
a filling scroll and. 
bottom member  is dual to a filling scroll. This makes 
the given Gaussian scroll $X_B$ itself
an $i$-th osculating scroll, $i\geq 0$,
 to the dual of a filling scroll. Thus, 
 every nondegenerate, non-cone Gaussian
 scroll belongs to a uniquely determined such Gaussian flag.
 In particular, we obtain a simple recipe for all
 Gaussian scrolls, namely:\par (F) start with a \ul {Filling} scroll,\par
 (D)  \ul {Dualize},\par
$(\d^i)$\ \ \ul {derive} $i$ times.  

\section{Some known constructions for stationary scrolls}\label{construction}
This is mainly from \cite{akivis-gauss}), slightly extended.
\subsection{Joins}\label{joins}
Let $Y_0,...,Y_r\subset\P^N$ and set 
\[X=\bigcup\{<p_0,...,p_r>:(p_0,...,p_r)\in Y_0\times...\times Y_r\}\]
This is clearly a Gaussian scroll as $\tilde T_{X, x}$ doesn't change as
$x$ moves generally in a fixed span $<p_0,...,p_r>$. Note that cones
are a special case of joins and joins are foliated by cones.
\subsection{Tangential (developable) and osculating varieties}\label{tangential}
For a variety $Y\subset\P^N$, let  $X$ be the closure of
the union of its embedded tangent spaces at smooth points.
Then $X$  determines a Gaussian scroll
because $\tilde T_{X, x}$ doesn't change as $x$ varies generally
in a fixed $\tilde T_{Y, y}$ and in fact equals the second-order
tangent space $T^{(2)}_{Y, y}$ (see \S \ref{stationary-gaussian}). 
Thus the Gauss deficiency of $X$ is at least
$\dim(Y)$. We can similarly construct Gaussian scrolls as 
$\tilde T^{(k)}_Y$, the
$k$-th osculating variety to $Y$, union of the
$k$-th osculating space $\tilde T^{(k)}_{Y, y}$.
The tangent space to $\tilde T^{(k)}_Y$ at a general point of 
$\tilde T^{(k)}_{Y, y}$ is $\tilde T^{(k+1)}_{Y, y}$.
\subsection{Inflation}\label{inflation}
 This is a slight generalization of 
the 'band' construction in \cite{akivis-gauss}.
Let \[X_0=\bigcup\limits_{b\in B}\P^k_b\]
be a stationary scroll with Gauss deficiency $k$
and suppose for each $b\in B$ we are given an inclusion
$\P^k_b\subset\P^\l_b$ 
i.e. a lifting of the classifying map of $X_0$ to the appropriate flag variety
parametrizing pairs $\P^k\subset\P^\l$ in $\P^N$.
Let 
\[X=\bigcup\limits_{b\in B}\P^\l_b\]
 and assume for simplicity
 $X$ is a scroll, which will be the case for
 general choices provided $\l-k<N-\dim(X_0)$.
$X$ is called an \emph{inflation} of $X_0$.  I claim that $X$
also has degenerate Gauss map, hence has the structure of a
stationary scroll (though not necessarily for the given
map $X\to B$). Indeed consider a general
open $\A^N\subset\P^N$.  
Consider $\l$ general sections $p_1,...,p_\l$ of $X\cap\A^N$ over $B$
such that $p_1,...,p_k\in X_0$.
For $x\in\A^\l_b=\P^\l_b\cap\A^N$, we can write
\[x=\sum\limits_{i=1}^\l \alpha_ip_i, \ \ p_1,...,p_k\in\A^k_b=\P^k_b\cap\A^N\]
where by fixing the $p_i$ and varying the $\alpha_i$, $x$ fills up
$\A^\l_b$ while $\sum\limits_{i=1}^k\alpha_ip_i$ fills up $\A^k_b$.
Then $\tilde T_{X, x}$ is generated $\mod\P^l_b$ by
$\del p_i/\del t_j, i=1,...,\l, j=1,...,\dim(B)$, where the $t_j$ are
coordinates on $B$, and this remains constant if $\alpha_1,...,\alpha_k$
are varied, thanks to $X_0$ being stationary. Note that the Gauss 
deficiency of $X$ is at least equal to that of $X_0$, i.e. $k$. 
$X$ is called an 
\emph{inflation} of $X_0$. In general, the stationary scroll corresponding to $X$
will have base $B'$ of dimension $\geq\dim(B)$ and fibre dimension
$k', k\leq k'\leq \l$. For the 'generic' inflation, we have $k'=k$.
\subsection{Base-change}\label{basechange}If $X/B$ is a stationary scroll
and $B'\to B$ is any morphism then the base-changed scroll $X_{B'}/B'$
is also a stationary scroll. In particular, if $B'\subset B$ is a curve
then $X_{B'}$ is a cone or developable by Griffiths-Harris (reproved below 
in Proposition \ref{1dim}).
Consequently $X$ can be foliated by such.
\subsection{Linearity of the fibres}
	For completeness we give a short proof of the classical result that that the general fibre
	component of a Gauss map is a linear space. 
	\begin{prop}\label{linear0}
	Let $Y\subset \P^N$ be
	an irreducible closed subvariety such that the (rational) Gauss map $\gamma:Y\cdots\to B$
	has poitive-dimensional general fibre. Then a general fibre component of $\gamma$ is
	a linear subspace of $\P^N$.
	\end{prop}
\begin{proof}	Filling up $B$ by curves such as 1-dimensional linear space sections $B_1$,
replacing $Y$ by $\gamma\inv(B_1)$,  and using 
\ref{basechange} above we may assume $B$ is a smooth curve. Let $Y'\to Y$ 
	be the normalization. Omitting a subset of codimension 2 we may assume $Y'$ is 
	smooth and the natural map $\gamma':Y'\to B$ is a nonconstant
	morphism. For $b\in B$
	general $F=(\gamma')\inv(b)$ and $L=\tilde T_yY$ for $y\in F$ general
	($L$ independent of $y$ for fixed $F$). Then setting  \[n=\dim (F),\] we have 
	\[n+1=\dim(Y)=\dim(L)\]
	We have an injection and generic isomorphism of rank-1 sheaves
	\[\O_F=N_{F/Y}\to N_{F/L}.\]
	If the inclusion $F\to L$ does  not go to a linear space (i.e. a hyperplane) then 
	the family of tangent spaces $\{T_zF:z\in F\}$ fills up $L$. But this means that a general
	$v\in L$ is actually tangent to $F$ somewhere so $\gamma'$ is infinitesimally constant in
	the direction $v$, hence $\gamma'$ is constant, which is a contradiction.
	\par
	\end{proof}

\section{Properties of Gauss maps and their fibres: a review}\label{properties-sec}
\subsection{Gauss maps and their fibres}
 In this section we review some
basic and well-known facts on Gauss maps.
Till further notice- which will come- $X$ denotes a closed
subvariety of $\P^N$ and $\gamma_X:X_{\mathrm{sm}}\to\gG(n, N)$ is its
Gauss map. With no loss or generality $X$ may be assumed nondegenerate
and 'dually', not a cone.
Note that a global vector field $\delta$ on $\P^N$ is induced 
by the action of the general linear group, which extends to
an action on the Grassmannian $\gG$ compatible with the
Pl\"ucker emebedding $\gG\to\P^M$. Thus we may view $\delta$
as a vector field on $\P^n\times\gG$ that extends to a vector
field on $\P^N\times\P^M$.\par
The Pl\"ucker line bundle on $\gG$ is by definition the pullback of $\O_{\P^M}(1)$
by the Pl\"ucker embedding, i.e. in the linear Grassmannian version $\gG(n,N)=G(n+1, N+1)$,
\[(A\subset\C^{N+1})\mapsto (\wedge^{n+1}A\subset\wedge^{n+1}\C^N).\]
A Pl\"ucker coordinate on $\gG$ is the restriction
of a linear (local) coordinate on $\P^M$.\par
Now let $F\subset X_\sm$ be a positive-dimensional component of 
the fibre of $\gamma_X$ over $[A]\in\gG$,
$A$ being an $n$-dimensional linear subspace in $\P^N$. If
$a\in A$ then $a$ extends (non-uniquely) to a global
vector field  $\delta_a$ on $\P^N$, thought of as a derivation,
which is everywhere tangent to $X$ along $F$. At the same time
$\delta$ also determines a tangent vector to the Gauss image
$\gamma_X(X_\sm)$ at $[A]$. We can trivialize the Pl\"ucker
line bundle $L$ in a neighborhood of $[A]$, hence
trivialize $\gamma^*(L)$ in a neighborhood of $F$.
I claim next that $F$ is not a multiple fibre of $\gamma_X$; 
equivalently, for a general Pl\"ucker coordinate $\phi$ vanishing
at $[A]$, $\gamma_X^*(\phi)$ vanishes to order 1 along $F$.
Indeed let 
\[\mu=\min\{\ord_F(\gamma_X^*(\phi)):\phi([A])=0\}\]
(minimum among Pl\"ucker coordinates). If $\mu>1$ then for
general $a\in A$, 
\[\ord_F(\delta_a(\gamma_X^*(\phi)))=
\ord_F(\gamma_X^*(\delta_a(\phi)))=\ord_F(\phi)-1.\]
However $\delta_a(\phi)$ is also a Pl\"ucker coordinate,
so this is a contradiction.\par
As we have seen in Proposition \ref{linear0} below and has been well known at least since \cite{griffiths-harris-ldg},
a general fibre $F$ of $\gamma_X$ is an open subset of a linear space
(see also \cite{landsberg-lectures}, \S 5 or   for another proof). 
Now I claim that for an \emph{arbitrary}
fibre $F$, not contained in the singular locus of $X$,
 the closure of $F$ must
meet the singular locus in codimension 1. Suppose not.
Then $F$ would contain a complete positive-dimensional subvariety $C$ 
disjoint from the singular locus
and the tangent sheaf to $X$ is globally generated
in a neighborhood of $C$. Integrating suitable
vector fields, we get a family of translates of $C$ filling
up an analytic neighborhood of $C$ in $X$, 
and each of these translates must be contained in some
fibre of $\gamma_X$. It follows that a general fibre of $\gamma_X$
must contain a translate of $C$ so we may as well assume $F$ is general.
Now if $F=\gamma_X\inv([L])$ then $N_{F/X}|_C\simeq N_{F/L}|_C$
because $X$ is everywhere tangent to $L$ along $F$.
But the former bundle is trivial, $F$ being a 
general fibre, while the latter bundle is ample in that $N_{F/L}(-1)$
is globally generated, which is a contradiction if $C$ is complete.
We have thus proven the following slight refinement of Proposition \ref{linear0}:
\begin{prop}
	A general fibre of $\gamma_X$ is an open subset in a linear space
	and the closure of an arbitrary fibre meets the singular locus of $X$ 
	in codimension 1.
	\end{prop}
\begin{example}[due to Mezzetti-Tommasi]
	It is possible for $\gamma_X$ to be generically finite but have isolated
	positive-dimensional fibres not contained in the singular locus. Moreover these
	need no be linear spaces. In fact E. Mezzetti and O. Tommasi (pers. comm.)
	point out the quartic surface in $\P^3$ with equation
	\[(x^2+y^2+z^2+3w^2)^2-16(x^2+y^2)w^2=0\]
	with singular locus $w=0$. The surface meets the plane $z=w$ in a double conic,
	which is a fibre of the (generically finite) Gauss map.
	\end{example}

\begin{example}
	In the $\P^5$ of plane conics, let $X$ be the hypersurface
	consisting of line-pairs, whose singular locus is the 
	Veronese surface $V$ 
	of double lines. At a point $L_1+L_2\in X\setminus V$,
	the tangent hyperplane to $X$ is the set of conics through
	$L_1\cap L_2$. Hence a fibre of the Gauss map is the set of all
	pairs of distinct lines through a fixed point, which is a
	$\P^2$ minus a conic.
	\end{example}
	\subsection{Second fundamental form} See \cite{griffiths-harris-ldg} for a detailed 
	presentation.
	We recall that given a smooth point $p$ on an $n$-dimensional variety $X\to\P^N$,
	the second fundamental form at $p$, denoted $\mathrm{II}$ or $\mathrm{II}_p$ or $\mathrm{II}_{p,X}$
	is the derivative at $p$ of the Gauss map $\gamma_X:X\to\gG:=\gG(n,N)$.
	As such, it is a map $\mathrm{II}:T_pX\to T_{\tilde T_pX}(\gG)$. Identifying $\P^N$ as $\P(\C^{N+1})$
	and denoting by $\hat T\subset\C^{N+1}$ the linear subspace corresponding to $\tilde T_pX$,
	we can view $\mathrm{II}$ as a map
	\[\mathrm{II}:T_pX\to \Hom(\hat T, \C^{N+1}/\hat T)=\Hom(T_pX, T_p\P^N/T_pX).\]
	Now $T_p\P^N/T_pX=N_pX$ is just the normal space at $p$, so we can view
	$\mathrm{II}$ as a bilinear form
	\[\mathrm{II}:T_pX\times T_pX\to N_pX\]
	It is well known (cf. \cite{griffiths-harris-ldg}, 1(b)) that this form is symmetric, i.e. a quadratic form, called
	the second fundamental form of $X$ at $p$. If we represent $X$ locally as a graph with equations
	\[y_{n+i}=f_i(x_1,...,x_n), i=1,...,N-n, f_i\in\m_p^2,\]
	then $\mathrm{II}_p$ is represented by the multi-matrix (vector of  symmetric martices)
	\[(\frac{\partial^2f_1}{\partial x_i\partial x_j}, ...,\frac{\partial^2f_{N-n}}{\partial x_i\partial x_j}).\]
\subsection{Gaussian scrolls: antiderived scroll}\label{stationary-leading}
Consider as in \S \ref{stationary-gaussian} as Gaussian (stationary) scroll $(X_B, f)$
with classifying map $\phi:B\to \gG(k, N)$ with associated derived scroll $X^{(1)}_B$.
For example, if $k=N-1$, $X_B$ just  corresponds to a 
generically finite map to its image in
the dual projective space:
\[ \phi:B\to\phi(B)\subset\gG(N-1, N)=\check\P^N.\]
In that case, the tangent developable to $\phi(B)$ (as 'plain' projective variety) 
corresponds in $\P^N$
to the 'leading edge' or 'cuspidal edge' or 'antiderived
scroll', denoted  $L(X/B)$ or $X^{(-1)}_B$.\par
We now extend this definition to the case of general fibre dimension $k$.
 Informally, denoting the fibres of $X/B$ by
$\P^k_b$, $L(X/B)=X_B^{(-1)}$ is the- possibly empty (!)-  scroll over $B$ with fibre
 at a general point
$b\in B$ equal to 
\eqspl{edge}{
	X^{(-1)}_b=\P^k_b\cap\bigcap\limits_{i=1}^g\partial\P^k_b/\partial t_i}
where $t_1,...,t_g$ are local coordinates at $b$, $g=\dim(B)$. An equivalent, more formal
definition is the following.
\begin{defn} With the above notations, set 
\[\P^k_b=\P(T), N=N_{P^k_b/X,x}=T_{X, x}/T_{\P^k_b, x},  x\in \P^k_b\] 
(independent of $x$), 
and let
\eqspl{nu}{\nu: T\to \Hom(T_bB, N)))
} be the normal map.
The the leading edge or antiderived scroll of $X/B$ is the sub-scroll  $X^{(-1)}_B=L(X/B)$
with fibres
\[X^{(-1)}_b=\P(\ker(\nu)).\]
\end{defn}
A natural way to describe the antiderived scroll is via
the dual projective space: if we denote by $X^\perp_B$ the dual $\P^{N-k-1}$ scroll in the dual
projective space $\check \P^N$, then
\eqspl{antiderived-eq}{X^{(-1)}_B=(X_B^\perp)^{(1)\perp}.}
\begin{rem}
When $k=N-1$,  it
 is clear e.g. by duality (see \S \ref{duality} below) that the
fibre dimension of $L(X/B)$ is $N-1-g$
which is the expected, but in general the fibre dimension
can exceed the expected (though it is always $<k$ provided $g>0$, thanks
to $c$ being generically finite).\end{rem}
	\begin{example}\label{curves}
	Given a nondegenerate curve $X\subset\P^N$, one has the associated osculating scrolls 
	$X^{(i)}_B, B=X, i<N$ which form a Gaussian flag. We have 
	$L(X_B^{(i)})=(X_B^{(i)})^{(-1)}=X^{(i-1)}_B$.
	\end{example}
A scroll $X/B$ is said to be \emph{filling} if the spreading map
$f:X\to\P^N$ is surjective. Since char. = 0 this is equivalent to the condition 
that at a general point $x\in X$ the derivative $d_xf$ is surjective.
The filling scrolls are exactly those whose dual $X^\perp_B$
has \emph{empty} leading edge  (see \S \ref{duality}).
\par
For future reference it is convenient to note the following, probably well-known
fact.
\begin{lem}\label{hyperplane}
	Let $X\subset\P^N$ be a variety with Gauss deficiency $k$.
	Then a general hyperplane section of $X$ has Gauss deficiency $k-1$.
	\end{lem}
\begin{proof}
	According to Griffiths-Harris \cite{griffiths-harris-ldg}, 
	(2.6) p. 387, the Gauss deficiency is characterized as the dimension
	of the kernel of the second fundamental form $\mathrm{II}_X$ at a general
	point. Now for a general hyperplane $H$ and a general point
	$p\in X\cap H$, we have
	\[\mathrm{II}_{X\cap H, p}=\mathrm{II}_X|_{T_p(X\cap H)}.\]
	From this the Lemma follows easily. 
	\end{proof}
Note that in our definition, the spreading map $f:X\to\P^N$
of a stationary scroll
is not assumed generically finite to its image, 
and in that case $X\to B$ may not be the actual
Gauss map of the image. Such scrolls, however, can be easily classified:
\begin{prop}
	(i) Let $X/B$ be a stationary scroll such that the spreading map
	$f:X\to \P^N$ is not generically finite to its image. Then there
	is a stationary scroll $X'/B'$ over a lower-dimensional
	base, together with a map $B\to B'$, such that for general $b'\in B'$,
	$X_{f\inv(b')}/f\inv(b')$ is a filling scroll of 
	the fibre projective space $X'_{b'}$.\par (ii) 
	Conversely given a stationary scroll $X'/B'$
	and a non-generically finite map $B\to B'$ together with a filling scroll
	$X/f\inv(b')$ for general $b'\in B'$, $X/B$ is a stationary scroll
	whose spreading map is not generically finite to its image. 
	\end{prop}
\begin{proof}
	By assumption the (well-defined)
	map $b\mapsto \im(df_x), x\in f\inv(b)$ factors through a lower-dimensional
	image $B'$ of $B$, which yields the scroll $X'/B'$. This scroll is stationary
	as it is just a descent via $B\to B'$ of the tangential scroll $X^{(1)}_B$.
	The converse is obvious.
	\end{proof}
\section{Gaussian flags: duality}\label{duality}
\subsection{Derived scrolls; Gaussian flags}
We recall from \S \ref{gaussian-flags-sec}  that a Gaussian flag in $\P^N$ has the form
\eqspl{forward}{X_B\subset X^{(1)}_B\subset...\subset X^{(\l)}_B,}
together with a spreading map $f:X^{(\l)}_B\to\P^N$ which we assume does not map onto $\P^N$
or into any proper linear subspace..
In this case it follows that the inclusions above are proper.
Thus any Gaussian scroll can be extended 'upwards' to a (uniquely determined)
maximal Gaussian flag containing $X_B$. 
This is called the osculating Gaussian flag associated to $X_B$.

\par
\subsection{Antiderived scrolls}\label{antiderived-sec}
Recall that we have defined the (possibly empty) antiderived scroll of
a Gaussian scroll $X_B$ as
\[(X_B)^{(-1)}=(X_B^\perp)^{(1)\perp}.\]
This again is Gaussian- this folows from the  fact it is proven in the Duality Theorem below that
$(X_B^{(-1)})^{(1)}=X_B$. Therefore $X_B^{(-1)}$  may be further antiderived etc., yielding what might be called
the anti-osculating flag of of $X_B$:

 \eqspl{downflag}{
 	X_B\supsetneq X^{(-1)}_B\supsetneq...\supsetneq X^{(-m)}_B
 } where $X^{(-m-1)}_B=\emptyset$.
 Now the Duality Theorem \ref{duality-thm} below will show that the anti-osculating flag 
 is in fact a Gaussian flag, 
hence the flags \eqref{forward} and \eqref{downflag} may be spliced together
to yield the (uniquely determined) \emph{maximal Gaussian flag},
extending the given stationary scroll $X_B$ both upwards and downwards. $m$ may be called the
(osculating) \emph{index} of $X_B$. Note that $X_B$- or for that matter
any $X_B^{(-i)}$- is a cone iff the bottom member $X_B^{(-m)}$ is a linear subspace.
   \par
   \subsection{Duality}
   
Now in general a Gaussian flag 
$X_B^{(0)}\subsetneq ...\subsetneq X^{(\l)}_B$
corresponds to a classifying map to a flag
variety \[c:B\to\fF(k_0,...,k_\l, N).\] Using the identification
\[\fF(k_\bullet, N)\simeq\fF(N-k_\bullet-1,  N),\]
a Gaussian flag $( X^{(\bullet)}_B)$ corresponds to a flag of scrolls
in the dual projective space $\check\P^N$ that we
call the \emph{linear dual}\footnote{The terminology
	is chosen to avoid confusion
	with dual variety.} flag and denote by
\[(X^{(\bullet)\perp}_B):
\ \ X^{(\l)\perp}_B\subsetneq...\subsetneq X^{(0)\perp}_B=X^\perp_B.\] 
This has the property that
\[ (X^{(i)\perp})_B^{(-1)}=X_B^{(i-1)}.\]
It is not a priori clear that the dual flag is Gaussian, but this is
a consequence of the Duality Theorem that we now state and prove:
\begin{thm}[Duality]\label{duality-thm}
	Assumptions as above, assume also that
	one (or equivalently every) $X_B^{(i)}$ is nondegenerate
	and not a cone. Then$(X_B^{(\bullet)\perp})$ is a Gaussian
	flag and its linear dual is $(X_B^{(\bullet)})$.
	\end{thm}
	\begin{rem}
	A trivial- but not entirely un-representative case of the duality theorem is:
	start with a variety $X$ with generically finite Gauss map and take
	$x\in X$ general  and let $N_1(x)$ be the first infinitesimal
	neighborhood of $x$ in $X$. Then
	\[\bigcap\limits_{x'\in N_1(x)}\tilde T_{x'}X=\{x\}.\]
	This is easy to check directly and will be proven in a more general form below
	(without the hypothesis of generically finite Gauss map).
	\end{rem}
\begin{proof}[Proof of Duality Theorem]
	The fact that the linear double dual   coincides with the original
	(i.e. $X_B^{\perp\perp}=X_B$) is just the simple linear algebra
	fact that $(Y^\perp)^\perp=Y$ for any subspace $Y$.
	The point is that $(X_B^{(\bullet)\perp})$
	is Gaussian, i.e. that
	\[((X_B^{(j+1)\perp})^{(1)}= (X_B^{(j)})^\perp.\]
	This is an assertion about the stationary scroll $X_B^{(j)}$ so to
	simplify notation set $Y_B=X_B^{(j)}$. Then 
	Then we have short flags
	\[Y_B\subset Y^{(1)}_B,\]
	\[ {Y_B^{(1)\perp(1)}}\supset Y^{(1)\perp}_B,\]
	and the assertion becomes that these are mutually dual,i.e. 
	that,
		for any stationary scroll $Y_B$, we have 
		\[Y_B=Y_B^{(1)\perp(1)\perp}\]
		or equivalently
	\eqspl{biduality}{(Y_B^{(1)\perp(1)}=Y_B^\perp.}
	Applying $\perp$ to \eqref{biduality}, an equivalent form is
	\[Y_B=(Y_B^{(1)})^{(-1)}.\]
	In other words, the leading edge of $Y^{(1)}_B$ is $Y_B$.
	Substituting $Y^\perp_B$ for $Y$ in \eqref{biduality}, yet another equivalent form is
	\eqspl{-1+1}{Y_B=(Y_B^{(-1)})^{(1)}.}
In other words, $Y_B^{(-1)}$- provided it is nonempty-
	is Gaussian with
	derived scroll $Y_B$.
		Note that because $Y_B$ is effective \eqref{-1+1}6 implies that
		$Y^{(-1)}_B$ is actually Gaussian.
	\par Now as for \eqref{biduality},
	note that the inclusion 
	\eqspl{oneway}{(Y_B^{(1)\perp})^{(1)}\subseteq Y_B^\perp} 
	or equivalently
	\eqspl{onewaydual}{
	(Y_B^{(1)\perp})^{(1)\perp}\supseteq Y_B} 

	is obvious from the definitions. Indeed
	pick a general point 
	$y\in Y$ and let $N_1(y)$ be its first-order neighborhood in $Y$
	and consider  the intersection of the tangent spaces $\tilde T_{y'}Y$
	as $y'$ ranges over $N_1(y)$, i.e. 
	$\bigcap\limits_{y'\in N_1(y)}\tilde T_{y'}Y\subset T_yY$ or,
	what is the same 
	$\bigcap\limits_v (\tilde T_yY\cap(\frac{\partial}{\partial v} 
		\tilde T_{y'}Y))$
	where $v$ ranges over $T_yY$ or any basis thereof. 
	What is 
	obvious is that if $F$ denotes the fibre of $Y/B$ through $y$,
	which is itself a $\P^k$, then
	the latter intersection contains $\tilde T_yF=F$, i.e.
	\[F\subseteq \bigcap\limits_{y'\in N_1(y)}\tilde T_{y'}Y\]
	(this is just \eqref{onewaydual}). The point in \eqref{biduality} is that the latter inclusion
	is an equality. Of course $\tilde T_{y'}Y$ depends only on the factored Gauss map
	$\bar\gamma$, so the equality in question becomes
	\eqspl{f}{F=\bigcap\limits_{b'\in N_1(b)}\bar\gamma(b').}
	Now the question is local at a general point $y\in F$, where locally $Y\to B$
	is just the Gauss map of the spread of $Y$ in $\P^N$, which may identified  with $Y$.
	Note that at the level of tangent spaces at $y$, the RHS of \eqref{f} is just
	the kernel of the second fundamental form $II_Y$, viewed and a map $T_yY\to\Hom(T_yY, N_yY)$.
	As $II_Y$ is the derivative of the Gauss  map $\gamma$ and $y$ is general, the kernel in question
	coincides with the tangent space to the fibre of $\gamma$, i.e. $F$. Therefore \eqref{f} holds.

	\end{proof}
\begin{rem}
	Another- actually not much different from that
	in \cite{griffiths-harris-ldg}- 
	proof of duality may be given as follows. First, using Lemma
	\ref{hyperplane}, we may assume that $X$ is a variety with generically finite
	Gauss map, hence nondegenerate second fundmental form $\mathrm{II}_X$. 
	What has to be proven is that the leading edge of $X^{(1)}$
	coincided with $X$. Generically projecting, 
	we may assume $X$ is a hypersurface, so $X^{(1)\perp}$
	is just a subvariety in the dual projective space $\check\P^N$. 
	Working locally and writing a general point of $X$ parametrically as $p(t)$,
	a point of $X^{(1)\perp}$ can be written
	parametrically as 
	\[p'(t)=\carets{p\wedge(\del p/\del t_1)\wedge...\wedge(\del p/\del t_n)}
	\in\check\P^N.\]
	We may assume the tangent vectors $\del/\del t_i$ are eigenvectors
	for the (nondegenerate, scalar-valued) second fundamental form $II_X$.
	Then
	\[\del p'/\del t_i=
	\carets{p\wedge(\del p/\del t_1)\wedge...\what{\del p/\del p_i}
		...\wedge(\del p/\del t_n)}\del^2 p/\del t_i^2, \]
	where $\del^2 p/\del t_i^2$ is a nonzero scalar by nondegeneracy of
	$\mathrm{II}_X$. This yields a nonzero linear form, i.e. hyperplane in 
	$\tilde T_pX$
	and because $p, \del p/\del t_1,...,\del p/\del t_n$ are a basis
	for $\tilde T_pX$, the intersection of these for $1=1,...,n$ is just $p$,
	which proves our assertion.
	\end{rem}
	\begin{example}\label{d-upple}
	Consider the $d$-uple Veronese surface $X\subset\P^N$, $N=(d+1)(d+2)/2-1$,
	 of degree $d^2$,
	which may be considered a trivial stationary scroll (fibre dimension 0).
	The tangent developable $X^{(1)}=\bigcup\limits_{p\in X}\tilde T_pX$ 
	is a stationary scroll of fibre dimension 2 with leading edge $X$.
Then $X^{(1)}$  is dual to the discriminant hypersurface $Y=\bigcup\limits_{p\in X}\tilde T_pX^\perp$
	of nodal plane cubics, 	of degree $3(d-1)^2$, which is a stationary scroll over $X$ with fibre dimension $N-3$,
	 and the tangent developable $Y^{(1)}$ is just dual to the original Veronese $X$. 
	 This is well known and illustrates the duality between tangent developable and leading edge.
	 \par Similar constructions can be made for any surface $X\subset\P^N$.
	\end{example}
	\subsection{Consequences of duality}
Now for a stationary scroll $X/B$ with osculating flag $X^{(\bullet)}_B$
let $n_i$ be the dimension of the image of $X^{(i)}_B$ in $\P^N$, 
which is also the fibre dimension of $X^{(i+1)}_B$ over $B$, so
that $\dim(X^{i+1)}_B)=g+n_i$. We call $(n_\bullet)$ the \emph{oculating
	dimension sequence} of the stationary scroll $X_B$.
\begin{cor}
	The assignment
	\[ X_B \leftrightarrow X^\perp_B\]
	 yields an idempotent bijection between Gaussian scrolls
	with with fibre dimension $k$
	$\P^N$
	and Gussian  scrolls with  fibre dimension 	$N-1-k$  in $\check\P^N$.
	Under this bijection $X^{(1)}_B$ corresponds to $(X^\perp_B)^{(-1)}$.
	\end{cor}
\begin{cor}\label{hypersurface}
	Let $X$ be a Gussian scroll of Gauss dimension $g$ such that
	$X$ (resp. the $i$-th osculating scroll of $X$) is a hypersurface.
	Then $X$ is linearly dual to a tangential scroll (resp. $(i+1)$-st 
	osculating scroll) of a $g$-dimensional variety.
	\end{cor}	
	\begin{example}[Elaboration]
	Let $\bar X\subset\P^N$ be a hypersurface with $g$ -dimensional Gauss image.
	It corresponds to a stationary scroll $X\to B$ with $\dim(B)=g$ and fibre $\P^{N-g-1}$.
	The derived scroll $X^{(1)}\to B$ has fibre $\P^{N-1}$ and 
	corresponds to a map $\phi:B\to\check\P^N$
	and $X$ is dual to the developable scroll $\phi(B)^{(1)}$. Conversely given
	a map $\phi:B\to\check\P^N$ (say generically finite to its nondegenerate image),
	$(\phi(B)^{(1)})^\perp\to\P^N$ will map to a hypersurface with $\dim(B)$-dimensional Gauss image.
	\end{example}
More generally,
\begin{cor}
Let $X$ be a Gaussian scroll of Gauss dimension $g$ in $\P^N$ such that
the $i$-th osculating scroll of $X$ has dimension $g+m<N$. Then $X$ is
linearly dual to the $i$th osculating of a stationary scroll
of dimension $g+N-m-1$ and Gauss dimension $\leq g$.

	\end{cor}
Coming back to Corollary \ref{hypersurface}, a general Gaussian scroll
of Gauss dimension $g$ and dimension $n$ in $\P^N$ may be projected
generically to a Gaussian hypersurface scroll 
$\pi(X)$ of Gauss dimension $g$  in $\P^{n+1}$. Then
that tangent scroll $\pi(X)^{(1)}$ is linearly dual to
a $g$-dimensional subvariety of $\check \P^{n+1}$ and
$\pi(X)^\perp$ is the tangent scroll of the latter.
Since projection is linearly dual to taking linear space section,
we conclude
\begin{cor} 
	If $X\subsetneq\P^N$ is an $n$-dimensional Gaussian scroll of Gauss 
	dimension $g$ then the section of $X^\perp$ by a general
	$(n+1)$-dimensional linear subspace in $\check\P^N$
	is a tangent scroll of a variety of dimension $g$. 
	\end{cor}
That is a sense in which a variety with degenerate Gauss map is
'built up from cones and tangent developables' (per Griffiths and Harris)
(the cones are exactly those $X$ so that $X^{(1)\perp}$ is degenerate). 
\begin{cor}
	Let $X_B$ be a nondegenerate, non-filling Gaussian scroll which is not a cone
	and whose antiderived scroll is nonempty. Then $X$ is a developable scroll.
	\end{cor}
\begin{proof}
	Our assumptions imply that the linear dual $X^\perp_B$ is non-filling and 
	nondegenerate. The antiderived $X^{(-1)}_B$ is just the dual of
	$X_B^{\perp(1)}$. Therefore by duality, $X_B=(X^{(-1)}_B)^{(+1)}_B$.
	
	\end{proof}
\begin{cor}
	Let $X_B$ be a nondegenerate Gaussian scroll whose antiderived
	is empty. Then $X_B$ is dual to a filling scroll that is not a cone.
	Conversely the dual to a filling scroll that is not a cone is
	a nondegenerate stationary scroll with empty antiderived.
	\end{cor}

\begin{cor}
	A nondegenerate, non-filling Gaussian scroll is dual to a tangent
	developable scroll
	\end{cor} 
\begin{proof}
	We may assume our nondegenerate non-filling Gaussian scroll $X_B$
	is not a cone in which case $X^\perp_B$ is nondegenerate,
	not a cone and has nonempty antiderived (namely $X_B^{(1)\perp}$);
the derived  scroll of the latter is $X_B^\perp$ itself.
	\end{proof}
Iterating the conclusion of the last Corollary, we conclude:
\begin{cor}
	A  Gaussian scroll that is not a cone is for some $i\geq 0$
	the $i$-th osculating scroll to the dual of a filling scroll. 
	\end{cor}
This yields a way to construct all stationary scrolls: start with a filling scroll,
dualize, then take some osculating scroll.
		
\begin{example}
	In $\P^5=\P(\Sym^2(\C^3))$, i.e. the space of conics in $\P^2$,
	 let $X$ be the Segre cubic primal
	consisting of all reducibles (line-pairs). It has Gauss dimension 2,
	and is linearly dual to the tangent scroll of the Veronese
	(2-uple embedding of $\P^2$). $X$ is also the tangent scroll
	of the dual Veronese of double lines. As in Example \ref{d-upple}, 
	the tangent scroll to $X$ is dual to the Veronese itself.
	\end{example}
\begin{example}
	Consider the example appearing in \cite{akivis-gauss}, IIIB, Class 1a.
	The is a 3-fold in $\P^4$ of Gauss dimension 2 that is a union of
	tangent lines in (one of two) eigendirections for the second fundamental form
	of a general surface $S$
	(assuming these eigendirections are distinct). Thus
	$X$ is a $\P^1$-bundle over a surface $B$ which can be locally identified with $S$.
	$X^{(1)}_B$ is a $\P^3$-bundle which fills up $\P^4$ and corresponds to a surface
	$X^{{(1)}\perp}_B\to\check\P^4$. The tangent developable of this surface
	is just $X^\perp_B$, a $\P^2$-bundle which surjects to $\check\P^4$.
	The dual to $X_B^{\perp(1)}$ is empty.
	\end{example}
\begin{rem}
	The dual to the 'inflation' construction of \ref{inflation}
	 may be called a 'deflation'. Thus, for a stationary scroll of fibre
	 dimension $k$, a deflation, i.e.
	 subscroll of codimension $\l\leq k$ is a staionary scroll
	 of fibre dimension $k-\l$ (which may be 0, so not correspond to a variety
	 with degenerare Gauss map, as noted in \cite{akivis-gauss}).
	\end{rem}
%

	\section{Varieties with small Gauss dimension}\label{small}

\begin{lem}
Suppose given the following:\par
 $U, V$,  finite-dimensional vector spaces,\par
 $L$, a line bundle on an irreducible projective variety $X$,\par
$i:U\otimes\O_X\to V\otimes L$ an injection
whose degeneracy locus
is a hypersurface numerically equivalent to $rL+A$ where $r=dim(U)$
and $A$ is nef.\par

Then $A$ is numerically trivial and
the saturation of	of $\im(i)$ is a constant subbundle, i.e.
has the form $V_0\otimes L$ for some subspace $V_0\subset V$.
Conversely if  $i$ is an injection
such that the saturation of $\im(i)$ 	is a constant subbundle
then the degeneracy locus of $i$ is numerically equivalent to $\dim(U)L$.
	\end{lem}
\begin{proof}
Note to begin with that a rank-$r$ 
subsheaf of $V\otimes L$ that is generically
contained in $V_0\otimes L$ where $\dim(V_0)=r$ saturates to $V_0\otimes L$.
Consequently we may cut $X$ down and assume it is a curve,
then normalize and assume it is nonsingular. Then the saturation
of $\im(i)$ has the form $E\otimes L$ where $E$ is a rank-$r$
subbundle of the trivial bundle $V\otimes \O_X$. Now $E$ corresponds to
a morphism $f$ of $X$ to the Grassmannian $G(r, V)$
and $-\det(E)$ is the pullback of the ample Pl\"ucker line bundle,
while the drop-rank locus of $i$, i.e. the zero-locus of
$\wedge^ri$, must be numerically equivalent to
$\det(E\otimes L)=\det(E)+rL$. 
Thus \[\det(E)+rL\equiv_{\mathrm{num}}rL+A,\]
i.e. $\det(E)$ is nef.
Consequently $\det(E)$ is numerically trivial
 so $f$ is constant and $E\otimes L=V_0\otimes L$ as claimed.
 The converse is trivial.
	\end{proof}
Classically, the \emph{focal locus} of a scroll  (originally due to Kummer \cite{kummer}; see also
\cite{semple-roth}, \cite{arrondo-focal}, \cite{depoi-1dimfocal})  
$\bigcup\limits_{b\in B}\P^k_b$ is the union of the intersections of 
fibres $\P^k_b$ with their 'consecutive' , i.e. the ramification locus of
the natural map
\[\coprod\limits_{b\in B}\P^k_b\to\bigcup\limits_{b\in B}\P^k_b,\]
or, in the  notation of \eqref{nu}, the projectivization of the
inverse image under $\nu$ of the discriminant hypersurface in
$\Hom(T_bB, N)$. More formally, the derivative of the classifying map
of the scroll is a map
\[T_bB\to\Hom(U, V/U)\]
where $U, V$ are the vector spaces corresponding to $\P^k_b$ resp. $\P^N$. Transposing,
we get a map
\[\psi:U\to\Hom(T_bB, V/U).\]
The focal locus is just the closure of the
projectivization $\P(\ker(\psi))$, spread out as a projective bundle
over some open subset of $B$.\par
Then the Lemma yields the following characterization of stationary scrools
among all scrolls in terms of the size of the focal locus:
\begin{cor}
	For a scroll of fibre dimension $k$,
	if the focal locus meets each fibre in
	a hypersurface then the degree of the hypersurface is
	at most $k$ with equality iff the scroll is stationary.
	\end{cor}
	\begin{thm}\label{lowdim}
	Let $X$ be a stationary scroll of dimension $n$ and Gauss dimension
	$g$ with $g(g+1)\leq n$. Then $X$ is an inflated cone or 
	inflated tangent scroll.
	\end{thm}
	\begin{proof}
	Let $L=\P^{n-g}_b$ be a general fibre of the Gauss map. The normal map
	\[\nu:T_bB\otimes \O_L\to N_{L/\P^N}\]
	factors through a constant subbundle $V\otimes\O_L(1)$
	where $\dim(V)=g$. Hence $\nu$ can be viewed as a 
	$g^2$-tuple of sections of $\O_L(1)$. 
	Since $g^2\leq n-g$ by assumption, there
	is a point $p\in L$ where $\nu$ vanishes. Viewing $p=p(t)$
	as describing a section $S$ over $B$ as t varies, 
	we get that $\del p/\del t_i\in L, \forall i$
	where the $t_i$ are local coordinates on $B$
	(the partials are well defined $\mod p$). This means $L$
	contains $\tilde T_pS$, so $X$ is an inflated cone 
	(if $p(t)$ is constant) or tangential variety (otherwise).
	\end{proof}

\begin{cor}
	Let $X_B$ be a nondegenerate stationary scroll of Gauss dimension $g$
	and codimension $c$ in $\P^N$, such that $g(g-1)<c$. 
	Then $X_B$ is a subscroll of the linear dual to a tangent
	scroll of a variety of dimension $g$ or less.
	\end{cor}
\begin{proof}
	Note that if the scroll $X/B\to\P^N$ has fibre dimension $n$ 
	and base dimension $g$, hence fibre dimension $n-g$,
	then $\X^\perp$ has dimension $g+(N-1-n+g)=2g+c-1$.
	Then our conclusion follows by 
	duality from Theorem \ref{lowdim} by noting
	that the linear dual of a nondegenerate variety cannot be a cone. 
	\end{proof}
In case $g=1$, Theorem \ref{lowdim} can be easily sharpened, 
a well-known result due to
Griffiths and Harris \cite{griffiths-harris}. Actually varieties of Gauss dimension
$g=2$ are also classified by results of Piontkowski \cite{piontkowski}.
\begin{prop}[Griffiths-Harris]\label{1dim}
	A variety with 1-dimensional Gauss image is  an osculating variety
	to a curve or a cone over such.
	\end{prop}
\begin{proof} The proof is by induction on $n$.
	Using notation as above, the zero-set of $\nu$ is a hyperplane
	$L_0\subset L=\P^{n-1}_b$. 
	If $L_0$ is fixed independent of $b\in B$ then our variety $X$ is a cone with
	vertex $L_0=\P^{n-2}$ and our assertion holds, so assume this is not the case.
	Let $X_0$ be the $n-1$-fold swept out
	by these hyperplanes as $b$ varies. Then for $x\in L_0$ general,
	clearly $\tilde T_{X_0, x}=L$ so $X_0$ is again a stationary
	scroll (or a curve, if $n=2$). So by induction we can write
	\[L_0=\carets{p(b), p'(b),...,p^{(k)}(b), c_1,...,c_{n-2-k}},\]
	for a variable point $p$ and constant points $c_i$ (possibly
	$k=-1$, so the point $p$ is not there and $L_0$ is constant).
	Hence
	\[L=\carets{p(b), p'(b),...,p^{(k+1)}(b), c_1,...,c_{n-2-k}},\]
	or possibly
	\[L=\carets{p(b), p'(b),...,p^{(k)}(b), c_1,...,c_{n-2-k}, c_{n-1-k}}\]
	completing the induction.
	\end{proof}

\begin{rem}
	Curiously, it follows from the above that the linear dual of
	an osculating scroll to a curve is also an 
	osculating scroll to a curve: however this 
	is well known and follows from the existence of a (classical)
	dual (not linear dual) 
	curve $\check X\subset\check\P^N$ which for $X$ nondegenerate
	is just the locus of $(N-1)$st osculating hyperplanes,
	given parametrically by $p\wedge p'\wedge...\wedge p^{(N-1)}$.
	\end{rem}
\begin{rem}
	The 'conjugate spaces' 
	construction in \cite{akivis-gauss}
	shows that for general $g$ analogues of Theorem \ref{lowdim}
	or Proposition \ref{1dim} are not available. 
	However varieties with Gauss dimension $g=2$ are classified by
	Piontkowski \cite{piontkowski}.
	\end{rem}
\begin{rem}\label{linear} Note that for any 
	Gaussian scroll $X/B$, the pullback of $X$ 
by  any curve
 $C\to B$ through a general point of $B$  is a variety $X_C/C$ to which Proposition \ref{1dim} applies.
 Consequently $X$ is in many ways a union of cones over osculating
 varieties to curves. 
 
 \end{rem} 
\bibliographystyle{amsplain}

\bibliography{../mybib}
\end{document}